\documentclass[12pt,reqno]{amsart}
\usepackage{amsmath,amsthm,amscd,amsfonts,amssymb,color}
\usepackage{cite}
\usepackage[mathscr]{eucal}
\usepackage{mathtools}

\usepackage[bookmarksnumbered,colorlinks,plainpages]{hyperref}
\newtheorem{theorem}{Theorem}[section]
\newtheorem{lemma}[theorem]{Lemma}
\newtheorem*{Acknowledgement}{\textnormal{\textbf{Acknowledgement}}}

\newtheorem{corollary}[theorem]{Corollary}
\newtheorem{note}[theorem]{Note}
\theoremstyle{definition}
\newtheorem{definition}[theorem]{Definition}
\newtheorem{example}[theorem]{Example}
\newtheorem{Open Prob}[theorem]{Open Problem}
\theoremstyle{remark}

\numberwithin{equation}{section}
\def\DJ{\leavevmode\setbox0=\hbox{D}\kern0pt\rlap{\kern.04em\raise.188\ht0\hbox{-}}D}
\begin{document}

\title[Best proximity point results in topological spaces]{Best proximity point results in  topological spaces and extension of Banach contraction principle}

\author[S.\ Som, S.\ Laha, L.K. \ Dey]
{Sumit Som$^{1}$, Supriti Laha$^{2}$, Lakshmi Kanta Dey$^{3}$}

\address{{$^{1}$\,} Sumit Som,
                    Department of Mathematics,
                    National Institute of Technology
                    Durgapur, India.}
                    \email{somkakdwip@gmail.com}

\address{{$^{2}$\,} Supriti Laha,
                    Department of Mathematics,
                    National Institute of Technology
                    Durgapur, India.}
                    \email{lahasupriti@gmail.com}

\address{{$^{3}$\,} Lakshmi Kanta Dey,
                    Department of Mathematics,
                    National Institute of Technology
                    Durgapur, India.}
                    \email{lakshmikdey@yahoo.co.in}

\begin{abstract}
In this paper, we introduce the notion of topologically Banach contraction mapping defined on an arbitrary topological space $X$ with the help of a continuous function $g:X\times X\rightarrow \mathbb{R}$ and investigate the existence of fixed points of such mapping. Moreover, we introduce two types of mappings defined on a non-empty subset of $X$ and produce sufficient conditions which will ensure 
the existence of best proximity points for these mappings. Our best proximity point results also extend some existing results from metric spaces or Banach spaces to topological spaces. More precisely,  our newly introduced mappings are more general than that of the corresponding notions introduced by Bunlue and Suantai [Arch. Math. (Brno), 54(2018),  165-176]. We present several examples to validate our results and justify its motivation.  To study best proximity point results, we introduce the notions of $g$-closed, $g$-sequentially compact subsets of $X$ and produce examples to show that there exists a non-empty subset of $X$ which is not closed, sequentially compact under usual topology but is $g$-closed and $g$-sequentially compact. 

\end{abstract}
\maketitle
\section{\bf{Introduction}}
Metric fixed point theory is an essential part of mathematics. It gives necessary and sufficient conditions that will ensure the existence of solutions of the equation $Ux=x$ where $U$ is a self-mapping defined on a metric space $ M$.  Such a solution is called a fixed point of the mapping $U.$ Banach contraction principle for standard metric spaces is a pioneer result in this connection. It has a lot of applications in the area of differential equations, integral equations. Over the years, many Mathematicians have weakened the metric structure and prove the Banach contraction principle for such spaces, but till now, no result is found, which extends the Banach contraction principle from metric spaces to arbitrary topological spaces. In this paper, we take an arbitrary topological space $X$ and a real-valued continuous function $g$ defined on the Cartesian product $X\times X.$ Then we introduce the notion of topologically Banach contraction mapping defined on a topological space $X$ with respect to $g$ and investigate  the existence of fixed points of such mapping. As a consequence, we extend the famous Banach contraction principle from standard metric spaces to topological spaces and we can retrieve the Banach contraction principle for metric spaces as a particular case of our theorem. On the other hand, if $U:A\rightarrow B$ is mapping where $A,B$ are non-empty subsets of the metric space $(M,\rho),A\neq B$ and $U(A)\cap A= \emptyset$ then the mapping $U$ has no fixed points. So, in case of a non-self map, one seek for an element in the domain space whose distance from its image is minimum i.e., the interesting problem is to $\mbox{minimize}~\rho(x,Ux)$ such that $x\in A.$ Since $\rho(x,Ux)\geq D(A,B)=\mbox{inf}~\{\rho(x,y):x\in A, y\in B\},$ so one can search for an element $x\in A$ such that $\rho(x,Ux)= D(A,B).$ Best proximity point problems deal with this situation. In the year $2011$, Basha \cite{BS} investigated the existence of best proximity points of proximal contractions. In the years $2013$ and $2015$, Gabeleh \cite{MG,MG1} introduced the notion of proximal nonexpansive mappings, Berinde weak proximal contractions in the context of metric spaces and investigated the existence of best proximity points of those classes of mappings. For more results, see \cite{MG2} and the references therein. In the year $2018$, Bunlue and Suantai \cite{BSS} introduced the notion of proximal weak contraction and proximal Berinde nonexpansive mappings and discussed the existence of best proximity points for those classes of mappings. All these results are formulated in the framework of metric spaces or Banach spaces where the standard metric or norm plays an important role. Recently in the year $2020$, Raj and Piramatchi \cite{RP} presented a way in which we can extend the best proximity point results from standard metric spaces to topological spaces, and it is exciting. In this paper, we have introduced the notions of topologically proximal weak contraction, topologically proximal Berinde non-expansive on topological spaces and discuss the existence of best proximity points for these mappings. We have presented ample examples to validate our results. Moreover, we have introduced these notions w.r.t a continuous function and present examples which show that there exist two continuous functions such that the mappings are topologically proximal weak contraction or topologically proximal Berinde non-expansive with respect to one continuous function but not with respect to another continuous function. Our best proximity point result about topologically proximal weak contractions also extends the Banach contraction principle for non-self mappings.

On the other hand, in 1970, Takahashi \cite{TA} first introduced the notion of convexity in metric spaces, and with the help of this notion, in this paper, we have defined the concept of topologically convex structure on an arbitrary topological space. Our best proximity results improve and extend the results in \cite{BS,BSS,MG,MG1} from standard metric spaces, Banach spaces to topological spaces. If the underlying space is metrizable with respect to the metric $d$, then by taking the continuous function $g=d$ we will recover those results. To build the theory of this paper, we introduce the notion of $g$-closed, $g$-sequentially compact subset of the topological space $X$ with the help of the continuous function $g$ defined on $X\times X$ and present examples to show that there exists a non-empty subset of $X$ which is not closed and sequentially compact with respect to the usual topology but is $g$-closed and $g$-sequentially compact for some continuous function $g.$ 

\section{\bf{Main results}}
We introduce some definitions which will be necessary for the development of our results.
\begin{definition}
Let $X$ be a topological space and $g: X\times X \rightarrow \mathbb{R}$ be a continuous function. Let $\{x_n\}$ be a sequence in $X$ and $x\in X.$ Then $\{x_n\}$ is said to be $g$-convergent to $x$ if $$|g(x_n, x)|\rightarrow 0 \ as\ n\rightarrow \infty$$ i.e., for a given $\varepsilon > 0$ there exists $k\in \mathbb{N}$ such that
$$|g(x_n, x)|< \varepsilon~\forall~ n\geq k.$$
\end{definition}

\begin{definition}
Let $X$ be a topological space and $g: X\times X \rightarrow \mathbb{R}$ be a continuous function. Let $\{x_n\}$ be a sequence in $X.$ Then $\{x_n\}$ is said to be $g$-Cauchy if $$|g(x_n, x_m)|\rightarrow 0~~\mbox{as}~\ n, m\rightarrow \infty$$ i.e., for a given $\varepsilon> 0$ there exists $ k\in \mathbb{N}$ such that
$$|g(x_n, x_m)|< \varepsilon~\forall~ n, m\geq k.$$
\end{definition}

\begin{definition}
Let $X$ be a topological space and $g: X\times X \rightarrow \mathbb{R}$ be a continuous function. $X$ is said to be $g$-complete if every $g$-Cauchy sequence $\{x_n\}$ is $g$-convergent to a point $x\in X.$
\end{definition}


\begin{lemma}\label{l1}
Let $X$ be a topological space and $g: X\times X \rightarrow \mathbb{R}$ be a continuous function such that $g(x,y)=0\Rightarrow x=y$ and $|g(x,z)|\leq |g(y,x)|+|g(y,z)|~\forall~x,y,z \in X.$ Then the limit of a $g$-convergent sequence is unique.
\end{lemma}

\begin{proof}
The proof is straightforward, so omitted.
\end{proof}

We show by an example that if the conditions of the Lemma \ref{l1} are violated, then the $g$-limit may not be unique.

\begin{example}\label{e2}
Consider $\mathbb{R}^2$ with usual topology. Define $g:\mathbb{R}^2 \times \mathbb{R}^2\rightarrow \mathbb{R}$ be defined by $g((x,y),(u,v))=xu.$ Then $g$ is a continuous function. Let $A=[-1,1]\times [-1,1].$ For the function $g$ we have, $g\Big((1,0),(0,1)\Big)=0$ but $(1,0)\neq (0,1).$ Also $g\Big((1,0),(1,0)\Big)=1\neq 0.$ So here, $g\Big((x,y),(u,v)\Big)=0 \nLeftrightarrow (x,y)=(u,v).$ Also, if we take $x=(1,0),y=(0,0),z=(4,0)$ then $|g(x,z)|>|g(y,x)|+|g(y,z)|.$ Now consider the sequence $\{x_n\}\subset A$ defined by $x_n=(\frac{1}{n},1).$ Then it can be seen that the sequence $\{x_n\}$ is $g$-convergent to $(0,1)$ and also to $(\frac{1}{2},1).$ So the limit is not unique. In fact, this sequence has infinitely many $g$-limits.
\end{example}

Now we introduce the notion of topologically Banach contraction mapping in a topological space $X$ with respect to a continuous function as follows:

\begin{definition}\label{d2}
Let $X$ be a topological space and $g:X\times X\rightarrow \mathbb{R}$ be a continuous function. Let $T:A\rightarrow B$ be a mapping where $A,B \subseteq X$ and $A,B \neq \phi.$ The mapping $T$ is said to be topologically Banach contraction w.r.t $g$ if there exists $\alpha\in (0,1)$ such that $$\Big|g\Big(T(x),T(y)\Big)\Big|\leq \alpha \Big|g(x,y)\Big|~\mbox{for all}~x,y \in A.$$
\end{definition}

We present an example of a topologically Banach contraction mapping $f$ with respect to a real valued continuous function $g,$ which is not a contraction mapping with respect to the metric $d$ with respect to which the space is metrizable. Also, in this example we show that, though, $f$ is a topologically Banach contraction mapping with respect to a real valued continuous function $g,$ may not be a topologically Banach contraction mapping with respect another real valued continuous function $h.$

\begin{example}
Consider $\mathbb{R}^{2}$ with the usual topology. Let $A=[\frac{1}{2},1]\times [0,1]$ and $B=[1,2]\times [0,1].$ Consider $g:\mathbb{R}^{2}\times \mathbb{R}^{2}\rightarrow \mathbb{R}$ by $g\Big((x_1,y_1),(u_1,v_1)\Big)=\mbox{min}~\{y_1,v_1\}.$ Define $T:A\rightarrow B$ by $T(x,y)=(2x,\frac{y}{2}),~(x,y)\in A.$ Let $x_1=(t_1,p_1)$ and $x_2=(t_2,p_2)\in A.$ Now $$\Big|g\Big(T(x_1),T(x_2)\Big)\Big|=\Big|g\Big((2t_1,\frac{p_1}{2}),(2t_2,\frac{p_2}{2})\Big)\Big|=\frac{1}{2}\Big|g(x_1,x_2)\Big|.$$
So, the mapping $T$ is topologically Banach contraction w.r.t $g.$ It can be seen that this mapping is not a contraction with respect to the usual metric on $\mathbb{R}^{2}.$

Now define $h:\mathbb{R}^{2}\times \mathbb{R}^{2}\rightarrow \mathbb{R}$ by $h\Big((x_1,y_1),(u_1,v_1)\Big)=x_1 u_1.$ Then $h$ is a continuous function. Now we will show that the mapping $T$ is not topologically Banach contraction w.r.t $h.$ Let $\alpha \in (0,1), x=(\frac{1}{2},0), y=(1,0).$ Now
$$\Big|h(T(x),T(y))\Big|=\Big|h((1,0),(2,0))\Big|=2>\alpha |h(x,y)|=\frac{\alpha}{2}.$$
So the mapping $T$ is not topologically Banach contraction w.r.t $h.$
\end{example}

Now, we give an example of a topologically Banach contraction mapping on a non-metrizable topological space.

\begin{example}
Consider $\mathbb{R}^{\omega},$ the space of all real valued sequences with the box topology. Then the space $\mathbb{R}^{\omega}$ is not metrizable. Define $g:\mathbb{R}^{\omega}\times \mathbb{R}^{\omega}\rightarrow \mathbb{R}$ by $$g\Big((x_n),(y_n)\Big)=x_1y_1;~(x_n),(y_n)\in \mathbb{R}^{\omega}.$$ Then $g$ is a continuous function on $\mathbb{R}^{\omega}\times \mathbb{R}^{\omega}$ with respect to the box topology. Define $f:\mathbb{R}^{\omega}\rightarrow \mathbb{R}^{\omega}$ by
$$f(x_1,x_2,\dots)=(0,x_1,x_2,\dots);~(x_1,x_2,\dots)\in \mathbb{R}^{\omega}.$$ Now let $x=(x_n),y=(y_n)\in \mathbb{R}^{\omega}$ and $\alpha\in (0,1).$ Then
$$0=\Big|g\Big(f(x),f(y)\Big)\Big|\leq \alpha \Big|g\Big(x,y\Big)\Big|.$$ So, $f$ is a topologically Banach contraction mapping w.r.t $g.$
\end{example}

We present our first desired result `Banach contraction principle' in a topological space.
\begin{theorem}\label{t2}
Let $X$ be a $g$-complete topological space where $g:X\times X\rightarrow \mathbb{R}$ is a continuous function such that $g(x,y)=0\Longrightarrow x=y,~   |g(x,y)|=|g(y,x)|,$ $|g(x,z)|\leq |g(x,y)|+|g(y,z)|$ for all $x,y,z\in X.$ Let $U:X\rightarrow X$ be a topologically Banach contraction mapping w.r.t $g$. Then $U$ has a unique fixed point and for any $p_0\in X,$ the sequence $p_{n+1}=U(p_n)$ for all $n\geq 0$ will converge to the unique fixed point of $U.$
\end{theorem}

\begin{proof}
As $U:X\rightarrow X$ is topologically Banach contraction map w.r.t $g$ so there exists $\beta \in (0,1)$ such that $$|g(U(x),U(y))|\leq \beta |g(x,y)|~\mbox{for all}~x,y \in X.$$ Let $p_0 \in X$ and define a sequence $\{p_n\}\subset X$ by $p_{n+1}=U(p_{n})~\mbox{for all}~n\geq 0, n\in \mathbb{N}.$ Now
\begin{eqnarray}
|g(p_{n+1},p_n)|&=& |g(U(p_n),U(p_{n-1})|\nonumber \\
&\leq& \alpha |g(p_n, p_{n-1})|\nonumber \\
&=& \alpha |g(U(p_{n-1}),U(p_{n-2})| \nonumber \\
& \vdots& \nonumber  \\
&\leq& \alpha^n |g(p_{1},p_{0})|. \nonumber
\end{eqnarray}
Suppose that $m>n$ and $n\in \mathbb{N}.$ Let $m=n+r$ where $r\geq 1.$ Now,
$$|g(p_n,p_{n+r})|\leq |g(p_n,p_{n+1})|+ |g(p_{n+1},p_{n+2})|+\dots +|g(p_{n+r-1},p_{n+r})|$$
$$\Longrightarrow |g(p_n,p_{n+r})|\leq \Big(\beta^{n}+\beta^{n+1}+\dots +\beta^{n+r-1}\Big)|g(p_0,p_1)|$$
$$\Longrightarrow |g(p_n,p_{n+r})|\leq \beta^{n}\Big(\frac{1-\beta^{r}}{1-\beta}\Big)|g(p_0,p_1)|\rightarrow 0~\mbox{as}~n,r\rightarrow \infty.$$
This shows that the sequence $\{p_n\}_{n\geq 0}$ is a $g$-Cauchy sequence. Since $X$ is $g$-complete, so the sequence $\{p_n\}_{n\geq 0}$ is $g$-convergent to $p^{*}\in X~\mbox{(say)}.$ Now since $U$ is topologically Banach contraction map so,
$$|g(U(p_n),U(p^{*}))|\leq \beta |g(p_n,p^{*})|\rightarrow 0~\mbox{as}~n\rightarrow \infty.$$
This shows that the sequence $\{U(p_n)\}_{n\geq 0}$ is $g$-convergent to $U(p^{*}).$ But~$p_{n+1}=U(p_n)$ so, $\{U(p_n)\}_{n\geq 0}$ is $g$-convergent to $p^{*}.$ As the continuous function $g$ satisfies the conditions of Lemma \ref{l1}, so the limit is unique. So, $U(p^{*})=p^{*}.$ So the mapping $U$ has a fixed point. Now suppose $U$ has two fixed points $p^{*}$ and $p^{**},p^{*}\neq p^{**}.$ Now,
$$|g(U(p^{*}),U(p^{**}))|\leq \beta |g(p^{*},p^{**})|$$
$$\Longrightarrow |g(p^{*},p^{**})|\leq \beta |g(p^{*},p^{**})|.$$ This is a contradiction since $|g(p^{*},p^{**})|>0.$ So the mapping $U$ has unique fixed point.
\end{proof}

\begin{note}
The preceding theorem is an extension of Banach contraction principle from metric space to general topological space $X$ with a continuous real-valued function $g$ defined on $X\times X.$ If the space $X$ is metrizable with respect to a metric $d$ then by taking $g=d,$ we will get the Banach contraction principle for standard metric spaces.
\end{note}

\begin{example}
Consider $X=[0,1]$ with the usual standard subspace topology inherit from $\mathbb{R}.$ Define $g:X\times X\rightarrow \mathbb{R}$ by $$g(x,y)=x^{2}-y^{2}.$$ Then $g$ is a continuous function on $\mathbb{R}\times \mathbb{R}.$ Define $T:X\rightarrow X$ by $$T(x)=\frac{x}{2}~, x\in X.$$ It can be seen that $[0,1]$ is $g$-complete. Here the continuous function $g$ satisfies all the conditions of Theorem \ref{t2}. But $g$ is not a metric, since $g$ can take negetive values. Let $x,y\in [0,1].$ Now
$$\Big|g(T(x),T(y))\Big|=\Big|g(\frac{x}{2},\frac{y}{2})\Big|=\frac{1}{4}|(x^{2}-y^{2}|)=\frac{1}{4}\Big|g(x,y)\Big|.$$ So $T$ is topologically Banach contraction map w.r.t $g$. By previous Theorem \ref{t2}, $T$ has unique fixed point. Here $p^{*}=0$ is the fixed point of $T.$
\end{example}

Now in the upcoming example, we will show that if any one of the conditions of the continuous function $g$ defined in Theorem \ref{t2} is violated, then there may exist infinitely many fixed points of the mapping $T.$

\begin{example}
Consider the space $\mathbb{R}^{2}$ with the usual topology. Define $g:\mathbb{R}^{2}\times \mathbb{R}^{2}\rightarrow \mathbb{R}$ by $$g\Big((x,y),(u,v)\Big)=y-v,~(x,y),(u,v)\in \mathbb{R}^{2}.$$ Then $g$ is a continuous function. Here $g\Big((1,2),(4,2)\Big)=0$ but $(1,2)\neq(4,2).$ So the function $g$ does not satisfied all the conditions of Theorem \ref{t2}. Now define $T:\mathbb{R}^{2}\rightarrow \mathbb{R}^{2}$ by $$T((x,y))=(x,\frac{y}{2}),~(x,y)\in \mathbb{R}^{2}.$$ Now we will show that $T$ is topologically Banach contraction map w.r.t $g.$ Let $(x_1,y_1),(x_2,y_2)\in \mathbb{R}^{2}.$ Now
$$\Big|g\Big(T(x_1,y_1),T(x_2,y_2)\Big)\Big|=\Big|g\Big((x_1,\frac{y_1}{2}),(x_1,\frac{y_1}{2})\Big)\Big|=\frac{1}{2}|y_1-y_2|=\frac{1}{2}\Big|g\Big((x_1,y_1),(x_2,y_2)\Big)\Big|.$$ So, $T$ is a topologically Banach contraction map w.r.t $g.$ Let $\{(x_n,y_n)\}$ be a $g$-Cauchy sequence in $\mathbb{R}^{2}.$ So $$\Big|g\Big((x_n,y_n),(x_m,y_m)\Big)\Big|\rightarrow 0~\mbox{as}~n,m\rightarrow \infty$$
$$\Longrightarrow |y_n-y_m|\rightarrow 0~\mbox{as}~n,m\rightarrow \infty.$$
So the sequence $\{y_n\}$ is a Cauchy sequence of real numbers. Let $y_n\rightarrow y\in \mathbb{R}~\mbox{as}~n\rightarrow \infty.$ Fix $n_0 \in \mathbb{N}.$ Now $$\Big|g\Big((x_n,y_n),(x_{n_0},y)\Big)\Big|=|y_n-y|\rightarrow 0~\mbox{as}~n\rightarrow \infty.$$ This shows that the sequence $\{(x_n,y_n)\}$ is $g$-convergent to $(x_{n_0},y).$ So $\mathbb{R}^{2}$ is $g$-complete. It can be seen that for any $x\in \mathbb{R},p^{*}=(x,0)$ is a fixed point of $T.$ So there are infinitely many fixed points of $T.$
\end{example}

\begin{corollary}
Let $X$ be a $g$-complete topological space where $g:X\times X\rightarrow \mathbb{R}$ is a continuous function such that $g(x,y)=0\Longrightarrow x=y,~|g(x,y)|=|g(y,x)|,$ $|g(x,z)|\leq |g(x,y)|+|g(y,z)|$ for all $x,y,z\in X.$ Let $U:X\rightarrow X$ be a mapping such that for some $n_0\in \mathbb{N},$ $U^{n_0}:X\rightarrow X$ is a topologically Banach contraction mapping w.r.t $g$. Then $U$ has a unique fixed point.
\end{corollary}

\begin{proof}
The proof follows from Theorem \ref{t2}, so omitted.
\end{proof}

Now we recall the following definition from \cite{RP}.

\begin{definition}\cite{RP}
Let $A$, $B$ be non-empty subsets of a topological space $X$. Let $g:X\times X \rightarrow \mathbb{R}$ be a continuous function. Define $$D_g(A,B) =  \inf\{|g(x, y)|: x\in A,\ y\in B\} .$$
\end{definition}

In this paper, we will use the following definitions.
$$A_g = \{x\in A: |g(x,y)|=D_g(A, B)~\mbox{for some}~ y\in B\} \mbox{ and}$$   $$B_g = \{y\in B: |g(x, y)|=D_g(A, B)~\mbox{for some}~ x\in A\}.$$

We like to introduce the definition of topologically proximal weak contraction in a topological space $X$ as follows:

\begin{definition}\label{d1}
Let $(A, B)$ be a pair of non-empty subsets of a topological space $X.$ A mapping $f:A\rightarrow B$ is said to be topologically proximal weak contraction with respect to a continuous function $g:X\times X\rightarrow \mathbb{R}$ if there exists $\beta\in (0, 1)$ and $ N\geq 0$ such that \[
\begin{rcases}
 |g(u_1, f(x_1))|= D_g(A, B)\\
 |g(u_2, f(x_2))|= D_g(A, B)
 \end{rcases}
  {\Longrightarrow |g(u_1,u_2)|\leq \beta |g(x_1,x_2)|+ N|g(x_2, u_1)|}
  \] for all $x_1, x_2, u_1, u_2 \in A.$
\end{definition}
\begin{note}
In Definition \ref{d1}, if we take $A=B$ and $N=0$ we may not get the Definition \ref{d2} because the continuous function $g$ does not necessarily satisfy $g(x,y)=0\nLeftrightarrow x=y$ as we already see in Example \ref{e2}.
\end{note}
\begin{note}
If the topological space $X$ is metrizable with respect to a metric $d$, then by taking $g=d$ we will get the notion of proximal weak contraction for standard metric spaces introduced by Bunlue and Suantai in \cite{BSS}. In particular, if we take $g=d$ and $N=0$ then we will get the notion of proximal contraction introduced by Basha in \cite{BS}.
\end{note}

In our last definition, we mention that the mapping $f$ is a topologically proximal weak contraction with respect to the continuous mapping $g$, and it is important. In our upcoming example we will show that there exist two subsets $A$ and $B$ in a topological space $X$ and a mapping $f:A\rightarrow B$ such that $f$ is topologically proximal weak contraction with respect to a continuous function $g$ but is not topologically proximal weak contraction with respect to another continuous function $h.$

\begin{example}
Consider $\mathbb{R}^2$ with the usual topology. Let $A=\{0\}\times [-1,1]$ and $B=\{1\}\times [-1,1].$ Let $T:A\rightarrow B$ be defined by $T(0,y)=(1,\frac{y}{4}).$ Let $g:\mathbb{R}^2 \times \mathbb{R}^2\rightarrow \mathbb{R}$ be defined by $g\Big((x,y),(u,v)\Big)=y^2-v^2.$ Then $g$ is a continuous function. Now we will show that $T$ is a topologically proximal weak contraction with respect to $g.$ It is clear that $D_g(A,B)=0.$ Let $x_1=(0,p_1),x_2=(0,p_2),u_1=(0,y_1),u_2=(0,y_2)\in A$ and $|g(x_1, T(u_1))|= 0$ and $|g(x_2, T(u_2))|=0.$ So $$\Big|g((0,p_1),(1,\frac{y_1}{4}))\Big|=0$$
$$\Longrightarrow p_1^{2}- \frac{y_1^{2}}{16}=0.$$
Similarly, from the second equation, we get,
$$p_2^{2}- \frac{y_2^{2}}{16}=0.$$
Now, $|g(x_1,x_2)|=p_1^{2}-p_2^{2}=\frac{1}{16}(y_1^{2}-y_2^{2})=\frac{1}{16}|g(u_1,u_2)|.$ This shows that $T$ is a topologically proximal weak contraction with respect to $g$ with $\beta=\frac{1}{16}$ and $N=0.$
Now let $h:\mathbb{R}^2 \times \mathbb{R}^2\rightarrow \mathbb{R}$ be defined by $h\Big((x,y),(u,v)\Big)=\mbox{min}\{y,v\}.$ It can be seen that $D_h(A,B)=0.$ Let $x_1=(0,\frac{1}{2}),x_2=(0,\frac{1}{4}),u_1=(0,0),u_2=(0,0)\in A$ and $|h(x_1, T(u_1))|= 0$ and $|h(x_2, T(u_2))|=0.$ Now if $\beta \in (0,1)$ and $N\geq 0$ then we have
$$\frac{1}{4}=|h(x_1,x_2)|>\beta |h(u_1,u_2)|+N|h(u_2,x_1)|=0.$$ This shows that $T$ is not topologically proximal weak contraction with respect to $h.$
\end{example}

In our next example, we show that the notion of topologically proximal weak contraction with respect to a continuous function is indeed more general than the notion of proximal weak contraction introduced by Bunlue and Suantai in \cite{BSS}. We show that, there exists a topological space $X$ with a continuous real valued function $g$, two non-empty disjoint subsets $A,B$ of $X$ and a function $f:A\rightarrow B$ such that $f$ is topologically proximal weak contraction w.r.t $g$ but if the topological space is metrizable with respect to a metric $d$ then $f$ is not proximal weak contraction w.r.t the metric $d.$
\begin{example}
Consider $\mathbb{R}$ with the usual topology. Let $g:\mathbb{R}\times \mathbb{R}\rightarrow \mathbb{R}$ be defined by 
$$g(x,y)=x^{2}-y^{2},~x,y\in \mathbb{R}.$$ Then $g$ is a continuous function. Let $A=\{0,1,2,3,5\}$ and $B=\{-1,-2,-3,4\}.$ Let $f:A\rightarrow B$ be defined by $f(0)=f(3)=f(5)=4, f(1)=-1, f(2)=-2.$ Then it can be seen that $D_g(A,B)=0.$ Let $\beta=\frac{1}{2}$ and $N=1.$ Now
$$\Big|g\Big(1,f(1)\Big)\Big|=D_g(A,B)$$
and
$$\Big|g\Big(2,f(2)\Big)\Big|=D_g(A,B).$$
Now $3=\Big|g\Big(1,2\Big)\Big|\leq \frac{1}{2}.\Big|g\Big(1,2\Big)\Big|+1.\Big|g\Big(2,1\Big)\Big|.$ This shows that $f$ is topologically proximal weak contraction w.r.t $g$ with $\beta=\frac{1}{2}$ and $N=1.$ Let $d$ denote the usual metric on $\mathbb{R}$ and $D(A,B)=\mbox{inf}\{d(x,y):x\in A, y\in B\}=1.$ Now 
$$d\Big(5,f(0)\Big)=D(A,B)$$
and
$$g\Big(0,f(1)\Big)=D(A,B).$$ But 
$$5=d(5,0)>\frac{1}{2}.d(0,1)+1.d(1,5).$$ So $f$ is not proximal weak contraction with respect to the usual metric on $\mathbb{R}$ with  $\beta=\frac{1}{2}$ and $N=1.$
\end{example}

In the following, we present a sufficient condition for topologically proximal weak contraction mappings to have a unique best proximity point in arbitrary topological space $X.$ Before that, we introduce the definition of a $g$-closed set and $g$-sequentially compact set as follows:

\begin{definition}
Let $X$ be a topological space and $g:X\times X\rightarrow \mathbb{R}$ be a continuous function. A non-empty subset $A$ of $X$ is said to be $g$-closed if every $g$-convergent sequence $\{x_n\}\subset A$, converges to a point in $A.$
\end{definition}

\begin{definition}
Let $X$ be a topological space and $g:X\times X\rightarrow \mathbb{R}$ be a continuous function. A non-empty subset $A$ of $X$ is said to be $g$-sequentially compact if every sequence $\{x_n\}_{n\in \mathbb{N}}$ in $A$ has a $g$-convergent subsequence $\{x_{n_{k}}\}$ which converges to a point in $A.$
\end{definition}

In the upcoming example we show that there exists a non-empty subset $A$ of topological space $X$ such that $A$ is $g$-closed but not closed in $X$ with respect to the usual topology. We also find a non-empty set which is $g$-sequentially compact but not sequentially compact with respect to the usual topology.

\begin{example}\label{e4}
Consider $\mathbb{R}$ with the usual topology and let $g:\mathbb{R}\times \mathbb{R}\rightarrow \mathbb{R}$ be defined by $g(x,y)=x-y+\frac{1}{2}.$ Let $A=(0,\infty).$ Then $A$ is not closed with respect to the usual topology in $\mathbb{R}.$ Let $\{x_n\}$ be a sequence in $A$ which is $g$-convergent to $x\in \mathbb{R}.$ So
$$\Big|g(x_n,x)\Big|\rightarrow 0~\mbox{as}~n\rightarrow \infty$$
$$\Longrightarrow \Big|x_n-x+\frac{1}{2}\Big|\rightarrow 0~\mbox{as}~n\rightarrow \infty$$
$$\Longrightarrow x_n\rightarrow (x-\frac{1}{2})~\mbox{as}~n\rightarrow \infty.$$
But since $\{x_n\}$ is a sequence in $(0,\infty)$ so, we have $x-\frac{1}{2}\geq 0.$ This shows that $x\geq \frac{1}{2}$ and $A$ is $g$-closed.

Now consider $\mathbb{R}^{2}$ with the usual topology and $g:\mathbb{R}^{2}\times \mathbb{R}^{2}\rightarrow \mathbb{R}$ be defined by
$$g\Big((x,y),(u,v)\Big)=y-v,~(x,y),(u,v)\in \mathbb{R}^{2}.$$
Then $g$ is a continuous function. Let $B=\{\frac{1}{n}: n\in \mathbb{N}\}\times \{0\}\cup \{\frac{1}{n}: n\in \mathbb{N}\}.$ Then $B$ is not sequentially compact in $\mathbb{R}^{2}$ with respect to the usual topology. If $\{x_n\}$ is a sequence in $B$ with finite range and $\{x_n\}$ is either constant or ultimately constant sequence then it is clear that there exists $p\in B$ such that $$\Big|g(x_n,p)\Big|\rightarrow 0~\mbox{as}~n\rightarrow \infty.$$ So, in this case, we take the subsequence as the sequence itself and $\{x_n\}$ is $g$-convergent to $p\in B.$ On the other hand, let $\{x_n\}$ is a sequence in $B$ with finite range and $(p,q)\in B$ is a cluster point of the sequence. In this case, we take the subsequence as, $x_{n_{k}}=(p,q)~\mbox{for all}~k\in \mathbb{N}.$ So, in this case $\{x_{n_{k}}\}$ is $g$-convergent to $(p,q)\in B.$  Now let $\{(p_n,t_n)\}$ be a sequence in $B$ with infinite range and $t_n\rightarrow 0~\mbox{as}~n\rightarrow \infty.$ In this case the sequence $\{(p_n,t_n)\}$ is $g$-convergent to $(1,0)\in B$ since
$$\Big|g\Big((p_n,t_n),(1,0)\Big)\Big|=\Big|t_n\Big|\rightarrow 0~\mbox{as}~n\rightarrow \infty.$$
Similarly if $\{(p_n,t_n)\}$ is a sequence in $B$ with infinite range and $p_n\rightarrow 0~\mbox{as}~n\rightarrow \infty$ then similarly we can show that there exists a subsequence of $\{(p_n,t_n)\}$ which is $g$-convergent to some element of $B.$ So, $B$ is $g$-sequentially compact in $\mathbb{R}^{2}$ but not sequentially compact.
\end{example}

In our next example we show that a non-empty subset, which is $g$-closed, where $g$ is a real valued continuous function on $X\times X$, may not be $h$-closed with respect to another real valued continuous function $h$ on $X\times X.$

\begin{example}
Consider $\mathbb{R}$ with the usual topology and let $g:\mathbb{R}\times \mathbb{R}\rightarrow \mathbb{R}$ be defined by $g(x,y)=x-y+\frac{1}{2}.$ Let $A=(0,\infty).$ Then from example \ref{e4}, $A$ is $g$-closed but $A$ is not closed with respect to usual topology on $\mathbb{R}.$ Now let $h:\mathbb{R}\times \mathbb{R}\rightarrow \mathbb{R}$ be defined by $h(x,y)=xy.$ Now we will show that the set $A$ is not $h$-closed. Consider the sequence $\{\frac{1}{n}\}$ in $A.$ Now the sequence $\{\frac{1}{n}\}$ is $h$-convergent to $-\frac{1}{2}$ since
$$\Big|h(\frac{1}{n},-\frac{1}{2})\Big|=-\frac{1}{2n}\rightarrow 0~\mbox{as}~n\rightarrow \infty$$ but $-\frac{1}{2}\notin A.$ So $A$ is not $h$-closed.
\end{example}


\begin{theorem}\label{t1}
Let $X$ be a $g$-complete topological space where $g:X\times X\rightarrow \mathbb{R}$ is a continuous function such that $g(x,y)=0\Leftrightarrow x=y, |g(x,y)|=|g(y,x)|$ and $|g(x,z)|\leq |g(x,y)|+|g(y,z)|$ for all $x,y,z\in X.$ Let $(A,B)$ be a pair of non-empty subsets of $X$ such that $A_g$ is non-empty and $g$-closed. Let $T:A\rightarrow B$ be topologically proximal weak contraction mapping w.r.t $g$ with $\beta \in (0,1),N\geq 0$ such that $T(A_g)\subseteq B_g.$ Then
\begin{enumerate}
 \item there exists a best proximity point $p*\in A_g$ of $T$ and the sequence $\{p_n\}_{n\geq 0}$ defined by $p_0 \in A_g$ and $|g(p_{n+1}, T(p_n))|= D_g(A, B)$ converges to the best proximity point of $T;$
 \item moreover if $(1-\beta-N)>0,$ then the best proximity point $p^{*}$ is unique.
\end{enumerate}
\end{theorem}

\begin{proof} Let $p_0 \in A_g.$ Since $T(A_g)\subseteq B_g$, we have $T(p_0)\in B_g.$ So, there exists $p_1 \in A_g$ such that $|g(p_1, T(p_0))|= D_g(A,B).$ Similarly, as $T(p_1)\in B_g$, so there exists $p_2\in A_g$ such that $|g(p_2, T(p_1))|= D_g(A,B).$ Continuing this process, we get a sequence $\{p_n\}_{n\geq 0} \subset A_g$ such that $$|g(p_{n+1}, T(p_n))|= D_g(A, B)~ \forall~n\geq 0.$$

Now we will show that the sequence $\{p_n\}_{n\geq 0}$ is a $g$-Cauchy sequence. From the construction of the sequence we have,
$$|g(p_n,T(p_{n-1}))|=D_g(A,B)$$ and $$|g(p_{n+1},T(p_n))|=D_g(A,B).$$ As $T$ is a topologically proximal weak contraction mapping w.r.t $g$, so we have,
$$|g(p_n,p_{n+1})|\leq \beta |g(p_{n-1},p_n)|+N |g(p_n,p_n)|$$
$$\Rightarrow |g(p_n,p_{n+1})|\leq \beta |g(p_{n-1},p_n)|.$$
So, we get $|g(p_n,p_{n+1})|\leq \beta^{n}|g(p_0,p_1)|.$ Suppose that $m>n$ and $n\in \mathbb{N}.$ Let $m=n+r$ where $r\geq 1.$ Now,
$$|g(p_n,p_{n+r})|\leq |g(p_n,p_{n+1})|+ |g(p_{n+1},p_{n+2})|+\dots +|g(p_{n+r-1},p_{n+r})|$$
$$\Longrightarrow |g(p_n,p_{n+r})|\leq \Big(\beta^{n}+\beta^{n+1}+\dots +\beta^{n+r-1}\Big)|g(p_0,p_1)|$$
$$\Longrightarrow |g(p_n,p_{n+r})|\leq \beta^{n}\Big(\frac{1-\beta^{r}}{1-\beta}\Big)|g(p_0,p_1)|\rightarrow 0~\mbox{as}~n,r\rightarrow \infty.$$
This shows that the sequence $\{p_n\}_{n\geq 0}$ is a $g$-Cauchy sequence. Since $X$ is $g$-complete, so the sequence $\{p_n\}_{n\geq 0}$ is $g$-convergent to a point $p^{*}\in X.$ Since $A_g$ is $g$-closed so $p^{*}\in A_g.$ Since $T(p^{*})\in B_g$ so there exists $x\in A_g$ such that $|g(x,T(p^{*})|=D_g(A,B).$ Also, $|g(p_{n+1},T(p_n))|=D_g(A,B).$ Thus we have
$$|g(p_{n+1},x)|\leq \beta |g(p_n,p^{*})|+N |g(p^{*},p_{n+1})|$$
$$\Longrightarrow |g(p_{n+1},x)|\rightarrow 0~\mbox{as}~n\rightarrow \infty.$$ This shows that the sequence $\{p_n\}_{n\geq 0}$ is also $g$-convergent to $x\in A_g.$ But since the limit is unique as we see from Lemma \ref{l1}, so, $x=p^{*}.$ We have $|g(p^{*},T(p^{*})|=D_g(A,B)$ that is, $p^{*}$ is a best proximity point of $T.$
Now, suppose the mapping $T$ has two best proximity points $p^{*}$ and $p^{**}.$ So we have
$$|g(p^{*},T(p^{*})|=D_g(A,B)$$ and
$$|g(p^{**},T(p^{**})|=D_g(A,B).$$
As $T$ is topologically proximal weak contraction, so we have,
$$|g(p^{*},p^{**})|\leq \beta |g(p^{*},p^{**})|+N |g(p^{*},p^{**})|$$
$$\Longrightarrow (1-\beta-N)|g(p^{*},p^{**})|\leq 0$$
$$\Longrightarrow |g(p^{*},p^{**})|=0 ~[\mbox{since}~(1-\beta-N)>0]$$
$$\Longrightarrow p^{*}=p^{**}~[\mbox{since}~g(x,y)=0\Rightarrow x=y].$$
So the best proximity point is unique.
\end{proof}
\begin{example}
Consider $\mathbb{R}^{2}$ with the usual topology and $X=\{1\}\times [-1,1]$ with subspace topology. Let $g:X\times X\rightarrow \mathbb{R}$ be defined by $$g\Big((x,y),(u,v)\Big)=y-v,~(x,y),(u,v)\in X.$$ Then $g$ is a continuous function on $X\times X.$ Now we will show that $X$ is $g$-complete. Let $\{(1,x_n)\}$ be a $g$-Cauchy sequence in $X.$ So
$$\Big|g\Big((1,x_n),(1,x_m)\Big)\Big|\rightarrow 0~\mbox{as}~n,m\rightarrow \infty$$
$$\Longrightarrow |x_n-x_m|\rightarrow 0~\mbox{as}~n,m\rightarrow \infty.$$
So the sequence $\{x_n\}$ is a Cauchy sequence in $[0,1].$ Since $[-1,1]$ is complete, so let $x_n\rightarrow p\in [-1,1]~\mbox{as}~n\rightarrow \infty.$ Now, $$\Big|g\Big((1,x_n),(1,p)\Big)\Big|=|x_n-p|\rightarrow 0~\mbox{as}~n\rightarrow \infty.$$ So the sequence $\{(1,x_n)\}$ is $g$-convergent to $(1,p)\in X.$ So $X$ is $g$-complete. Now let $A=\{1\}\times [-1,0]$ and $B=\{1\}\times [0,1].$ Then $D_g(A,B)=0.$ Now let $(1,x)\in A_g.$ Then there exists $(1,y)\in B$ such that $|g((1,x),(1,y))|=0.$ So $|x-y|=0.$ This is satisfied only by $x=0.$ This shows that $A_g=\{(1,0)\}.$ Also, $B_g=\{(1,0)\}.$  So, $A_g$ is non-empty and $g$-closed. Now it can be seen that the function $g$ is satisfied all the conditions of Theorem \ref{t1}.

Now define $f:A\rightarrow B$ by $$f(1,x)=(1,-\frac{x}{2}),~(1,x)\in A.$$ So, $f(1,0)=(1,0)\Longrightarrow f(A_g)\subseteq B_g.$ Let $(1,p_1),(1,p_2),(1,u_1),(1,u_2)\in A$ such that $$\Big|g\Big((1,p_1),f(1,u_1)\Big)\Big|=0$$ and  $$\Big|g\Big((1,p_2),f(1,u_2)\Big)\Big|=0.$$ These two equations imply that $p_1+\frac{u_1}{2}=0$ and $p_2+\frac{u_2}{2}=0.$ Now
$$\Big|g\Big((1,p_1),(1,p_2)\Big)\Big|=|p_1-p_2|=\frac{1}{2}|u_1-u_2|=\frac{1}{2}\Big|g\Big((1,u_1),(1,u_2)\Big)\Big|.$$ This shows that the mapping $f$ is topologically proximal weak contraction mapping with respect to $g.$ So all the conditions of Theorem \ref{t1} are satisfied. So, by the Theorem \ref{t1} the mapping $f$ has a best proximity point in $A_g.$ Here $p^{*}=(1,0)\in A_g$ is a best proximity point of $f.$ In this example the best proximity point is unique because here $(1-\beta-N)=\frac{1}{2}>0.$ So the Theorem \ref{t1} holds good.
\end{example}

\begin{note}
Theorem \ref{t1} is an extension and improvement of Theorem $3.1$ of \cite{BSS} from metric space to topological space $X$ with a continuous function defined on $X\times X.$ If the topological space $X$ is metrizable with respect to metric $d$ then by taking $g=d$ we will get Theorem 3.1 of \cite{BSS}. Also in Theorem \ref{t1} if we take $A=B,g=d$ and $N=0$ then we get the Banach contraction principle for topological spaces. So Theorem \ref{t1} is an extension of Banach contraction principle from standard metric spaces to topological spaces.
\end{note}

In a topological space $X$ on which there is no linear space structure, it is hard to define the notion of convex sets in $X.$ In order to define the notion of convex sets in an arbitrary topological space $X$ with a continuous function $g:X\times X\rightarrow \mathbb{R}$, we first introduce the notion of topologically convex structure on $X$ as follows:

\begin{definition}
Let $X$ be a topological space and $g:X\times X\rightarrow \mathbb{R}$ be a continuous function. A continuous function $H:X\times X\times [0,1]\rightarrow X$ is called topologically convex structure w.r.t $g$ if the two conditions are satisfied:
\begin{enumerate}
 \item $|g(x_0,H(x,y,\lambda))|\leq \lambda |g(x_0,x)|+(1-\lambda)|g(x_0,y)|~\mbox{for all}~x_0,x,y\in X~\mbox{and}~\lambda\in [0,1];$
 \item $|g(H(x,y,\lambda),H(x_0,y_0,\lambda))|\leq \lambda |g(x,x_0)|+(1-\lambda)|g(y,y_0)|~\mbox{for all}~x_0,x,y,y_0\in X~\mbox{and}~\lambda\in [0,1].$
\end{enumerate}
\end{definition}

A non empty subset $A$ of $X$ is said to be convex if $H(x,y,\lambda)\in A~\mbox{for all}~x,y \in A~\mbox{and}~\lambda\in [0,1].$ Now by using the notion of topologically convex structure, we will define the concept of topologically $r$-starshaped subset of $X$ as follows:

\begin{definition}
Let $X$ be a topological space and $g:X\times X\rightarrow \mathbb{R}$ be a continuous function. Let $H:X\times X\times [0,1]\rightarrow X$ be a topologically convex structure on $X$ w.r.t $g.$ A non empty subset $A$ of $X$ is called topologically $r$-starshaped if there exists a point $r\in A$ such that $H(r,x,\lambda)\in A~\mbox{for all}~x\in A~\mbox{and}~\lambda\in [0,1].$
\end{definition}

Now we present a lemma that will be necessary for our upcoming theorem about best proximity points.

\begin{lemma}\label{l2}
Let $X$ be a topological space and $g:X\times X\rightarrow \mathbb{R}$ be a continuous function. Let $H:X\times X\times [0,1]\rightarrow X$ be a topologically convex structure on $X$ w.r.t $g.$ Let $A,B(\neq \phi)\subset X$ such that $A$ is topologically $r$-starshaped and $B$ is topologically $s$-starshaped and $|g(r,s)|=D_g(A,B).$ Then $A_g$ is topologically $r$-starshaped and $B_g$ is topologically $s$-starshaped.
\end{lemma}

\begin{proof}
Since $|g(r,s)|=D_g(A,B)$ so $r\in A_g$ and $A_g\neq \phi.$ Let $x\in A_g~\mbox{and}~\lambda\in [0,1].$ Since $x\in A$ and $A$ is topologically $r$-starshaped, so $H(r,x,\lambda)\in A.$ Since $x\in A_g$ so there exists $y\in B_g$ such that $|g(x,y)|=D_g(A,B).$ Since $y\in B$ and $B$ is topologically $s$-starshaped, so $H(s,y,\lambda)\in B.$ Now
$$ D_g(A,B)\leq \Big|g(H(r,x,\lambda),H(s,y,\lambda))\Big|\leq \lambda |g(r,s)|+(1-\lambda)|g(x,y)|=D_g(A,B).$$
This shows that $H(r,x,\lambda)\in A_g$ and $A_g$ is topologically $r$-starshaped. Similarly we can show that $B_g$ is topologically $s$-starshaped.
\end{proof}

Now we introduce the concept of topologically semi-sharp proximinal pair in an arbitrary topological space $X$ as follows:

\begin{definition}
Let $(A, B)$ be a pair of non-empty subsets of a topological space $X$. The pair $(A, B)$ is said to be a topologically semi-sharp proximinal pair w.r.t a continuous function $g:X\times X\rightarrow \mathbb{R}$ if for each $x\in A$ there exists at most one $x^*$ in $B$ such that $$|g(x,x^*)|=D_g(A, B).$$
\end{definition}

\begin{lemma}\label{l3}
Let $X$ be topological space and $g:X\times X\rightarrow \mathbb{R}$ be a continuous function. Let $(A,B)(\neq \phi)\subset X$ be a topologically semi-sharp proximinal pair w.r.t $g$ such that $A_g, B_g \neq \phi.$ Then $(A_g,B_g)$ is a topologically semi-sharp proximinal pair w.r.t $g.$
\end{lemma}

\begin{proof}
The proof is straightforward, so omitted.
\end{proof}

Now we like to introduce the notion of topologically proximal Berinde non-expansive mapping in a topological space $X$ as follows:

\begin{definition}
Let $(A, B)$ be a pair of non-empty subsets of a topological space $X.$ Let $g:X\times X\rightarrow \mathbb{R}$ be a continuous function. A mapping $f:A\rightarrow B$ is said to be topologically proximal Berinde non-expansive w.r.t $g$, if there exists $N\geq 0$ such that \[
\begin{rcases}
 |g(u_1, f(x_1))|= D_g(A, B)\\
 |g(u_2, f(x_2))|= D_g(A, B)
 \end{rcases}
  {\Longrightarrow |g(u_1,u_2)|\leq |g(x_1,x_2)|+ N|g(x_2, u_1)|}
  \] for all $x_1, x_2, u_1, u_2 \in A.$
\end{definition}

\begin{note}
If the topological space $X$ is metrizable with respect to a metric $d$, then by taking $g=d$ we will get the notion of proximal Berinde nonexpansive mappings for standard metric spaces introduced by Bunlue and Suantai in \cite{BSS}. In particular, if we take $g=d$ and $N=0$ then we will get the notion of proximal nonexpansive mappings introduced by Gabeleh in \cite{MG}.
\end{note}

In our last definition, we mention that the mapping $f$ is a topologically proximal Berinde non-expansive w.r.t the continuous mapping $g$, and it is important. In our upcoming example, we will show that there exist two subsets $A$ and $B$ in a topological space $X$ and a mapping $f:A\rightarrow B$ such that $f$ is topologically proximal Berinde non-expansive w.r.t a continuous function $g$ but is not topologically proximal Berinde non-expansive w.r.t another continuous function $h.$

\begin{example}
Consider $X=[0,3]\times\mathbb{R}$ with the usual subspace topology of $\mathbb{R}^{2}\times \mathbb{R}^{2}.$ Let $A=\{0\}\times \mathbb{R}$ and $B=\{3\}\times \mathbb{R}.$ Let $T:A\rightarrow B$ be defined by $T(0,y)=(3,y).$ Let $g:\mathbb{R}^2 \times \mathbb{R}^2\rightarrow \mathbb{R}$ be defined by $g((x,y),(u,v))=y^2-v^2.$ Then $g$ is a continuous function. Now we will show that $T$ is a topologically proximal Berinde non-expansive w.r.t $g.$ It is clear that $D_g(A,B)=0.$ Now let $x_1=(0,p_1),x_2=(0,p_2),u_1=(0,y_1),u_2=(0,y_2)\in A$ and $|g(x_1, T(u_1))|= 0$ and $|g(x_2, T(u_2))|=0.$ So $$\Big|g((0,p_1),(3,y_1))\Big|=0$$
$$\Longrightarrow p_1^{2}- y_1^{2}=0.$$
Similarly, from the second equation, we get,
$$p_2^{2}- y_2^{2}=0.$$
Now, $|g(x_1,x_2)|=|p_1^{2}-p_2^{2}|=|y_1^{2}-y_2^{2}|=|g(u_1,u_2)|.$ This shows that $T$ is a topologically proximal Berinde non-expansive w.r.t $g$ with $N=0.$

Now let $h:\mathbb{R}^2 \times \mathbb{R}^2\rightarrow \mathbb{R}$ be defined by $h((x,y),(u,v))=\mbox{min}\{y,v\}.$ It can be seen that $D_h(A,B)=0.$ Let $x_1=(0,1),x_2=(0,2),u_1=(0,0),u_2=(0,0)\in A$ and $|h(x_1, T(u_1))|= 0$ and $|h(x_2, T(u_2))|=0.$ Now if $N\geq 0$ then we have
$$ 1=|h(x_1,x_2)|> |h(u_1,u_2)|+N|h(u_2,x_1)|=0.$$ This shows that $T$ is not topologically proximal Berinde non-expansive w.r.t $h.$
\end{example}

In our next example, we show that the notion of topologically proximal Berinde non-expansive mapping with respect to a continuous function is indeed more general than the notion of proximal Berinde non-expansive mapping introduced by Bunlue and Suantai in \cite{BSS}. We show that, there exists a topological space $X$ with a continuous real valued function $g$, two non-empty disjoint subsets $A,B$ of $X$ and a function $f:A\rightarrow B$ such that $f$ is topologically proximal Berinde non-expansive w.r.t $g$ but if the topological space is metrizable with respect to a metric $d$ then $f$ is not proximal Berinde non-expansive w.r.t the metric $d.$
\begin{example}
Consider $\mathbb{R}$ with the usual topology. Let $g:\mathbb{R}\times \mathbb{R}\rightarrow \mathbb{R}$ be defined by 
$$g(x,y)=x^{2}-y^{2},~x,y\in \mathbb{R}.$$ Then $g$ is a continuous function. Let $A=\{0,1,2,3,5\}$ and $B=\{-1,-2,-3,4\}.$ Let $f:A\rightarrow B$ be defined by $f(0)=f(3)=f(5)=4, f(1)=-1, f(2)=-2.$ Then it can be seen that $D_g(A,B)=0.$ Let $N=1.$ Now
$$\Big|g\Big(1,f(1)\Big)\Big|=D_g(A,B)$$
and
$$\Big|g\Big(2,f(2)\Big)\Big|=D_g(A,B).$$
Now $3=\Big|g\Big(1,2\Big)\Big|\leq \Big|g\Big(1,2\Big)\Big|+1.\Big|g\Big(2,1\Big)\Big|.$ This shows that $f$ is topologically proximal Berinde non-expansive w.r.t $g$ with $N=1.$ Let $d$ denote the usual metric on $\mathbb{R}$ and $D(A,B)=\mbox{inf}\{d(x,y):x\in A, y\in B\}=1.$ Now 
$$d\Big(0,f(1)\Big)=D(A,B)$$
and
$$g\Big(3,f(0)\Big)=D(A,B).$$ But 
$$3=d(0,3)>d(1,0)+1.d(0,0).$$ So $f$ is not proximal Berinde non-expansive with respect to the usual metric on $\mathbb{R}$ with $N=1.$
\end{example}

We will present a theorem regarding the existence of best proximity point of a topologically proximal Berinde non-expansive mapping in topological spaces.

\begin{theorem}\label{t3}
Let $X$ be a $g$-complete topological space where $g:X\times X\rightarrow \mathbb{R}$ is a continuous function such that $g(x,y)=0\Leftrightarrow x=y,~   |g(x,y)|=|g(y,x)|,$ $|g(x,z)|\leq |g(x,y)|+|g(y,z)|$ for all $x,y,z\in X$ and $|g(r,x)|+|g(y,s)|=2D_g(A_g,B_g)(A_g,B_g \neq \phi)~\mbox{for all}~x\in B_g,~y\in A_g.$ Let $H:X\times X\times [0,1]\rightarrow X$ be a topologically convex structure on $X$ w.r.t $g.$ Let $(A, B)$ be a topologically semi-sharp proximinal pair of non-empty subsets of $X$ such that $A$ is $r$-starshaped, $B$ is $s$-starshaped  w.r.t $g$ and $|g(r,s)|=D_g(A, B).$ Assume $A_g$ is compact, $g$-sequentially compact and $g$-closed. Suppose that $f: A\rightarrow B$ satisfies the following conditions:\begin{enumerate} \item $f$ is a topologically proximal Berinde non-expansive mapping with respect to $g;$ \item $f(A_g)\subseteq B_g.$ \end{enumerate} Then there exists $p^{*} \in A_g$ such that $|g(p^{*}, f(p^{*}))|= D_g(A, B)$ i.e. $f$ has a best proximity point in $A_g.$
\end{theorem}

\begin{proof}
Let $x\in A_g.$ Since $f(A_g)\subseteq B_g,$ so $f(x)\in B_g.$ As $(A, B)$ is a topologically semi-sharp proximinal pair of non-empty subsets of $X$ such that $A$ is $r$-starshaped, $B$ is $s$-starshaped  w.r.t $g$ then by Lemma \ref{l2}, we have $A_g$ is $r$-starshaped, $B_g$ is $s$-starshaped  w.r.t $g.$ Now define the sequence of functions $f_n:A_g\rightarrow B_g$ by $$f_{n}(x)=H(s,f(x),a_n),~x\in A_g.$$ Here the sequence $(a_n)\subset (0,1)$ is such that $a_n\rightarrow 0$ as $n\rightarrow \infty.$ Now we will show that the sequence of functions $\{f_n\}$ is topologically proximal weak contraction for each $n\in \mathbb{N}.$ Let $p_1,p_2,q_1,q_2\in A_g$ such that
$$\Big|g(p_1,f_{n}(q_1))\Big|=D_g(A_g,B_g)$$ and
$$\Big|g(p_2,f_{n}(q_2))\Big|=D_g(A_g,B_g).$$
As $f(q_1),f(q_2)\in B_g$ so there exist $r_1,r_2 \in A_g$ such that $$|g(r_1,f(q_1))|=D_g(A,B)$$ and $$|g(r_2,f(q_2))|=D_g(A,B).$$  Since $f$ is topologically proximal Berinde non-expansive map so we have,
\begin{equation}\label{e1}
|g(r_1,r_2)|\leq |g(q_1,q_2)|+ N|g(q_2,r_1)|,~N\geq 0 .
\end{equation}
From Lemma \ref{l2} and Lemma \ref{l3}, we have the sets $A_g$ and $B_g$ are topologically $r$-starshaped and topologically $s$-starshaped respectively and $(A_g,B_g)$ is a topologically semi-sharp proximinal pair of the topological space $X.$ So, $H(r,r_1,a_n)\in A_g$ and $H(r,r_2,a_n)\in A_g.$ Now
\begin{eqnarray}
D_g(A_g,B_g)&\leq& |g(H(r,r_1,a_n),f_{n}(q_1))|
= |g(f_{n}(q_1), H(r,r_1,a_n))| \nonumber\\
&\leq& a_n|g(f_{n}(q_1),r)|+(1-a_n)|g(f_{n}(q_1),r_1)|\nonumber \\
&=& a_n|g(H(s,f(q_1),a_n),r)|+(1-a_n)|g(H(s,f(q_1),a_n),r_1)| \nonumber \\
&\leq& a_n\Big\{a_n |g(r,s)|+(1-a_n)|g(r,f(q_1)|\Big\}\nonumber \\
&+&(1-a_n)\Big\{a_n |g(r_1,s)|+(1-a_n)|g(r_1,f(q_1)|\Big\}\nonumber \\
&=& \Big\{a_n^{2}+(1-a_n)^{2}\Big\}D_g(A,B)+a_n(1-a_n)\Big\{|g(r,f(q_1)|+|g(r_1,s)|\Big\}\nonumber \\
&\leq& D_g(A_g,B_g).\nonumber
\end{eqnarray}
$$\Longrightarrow |g(H(r,r_1,a_n),f_{n}(q_1))|=D_g(A_g,B_g).$$
Similarly, we can show that
$$|g(H(r,r_2,a_n),f_{n}(q_2))|=D_g(A_g,B_g).$$
Since $(A_g,B_g)$ is a topologically semi-sharp proximinal pair, so we have
$$ p_1= H(r,r_1,a_n), p_2=H(r,r_2,a_n).$$
Since $A_g$ is a compact topological space so $A_g \times A_g$ is compact. Since $g:X\times X\rightarrow \mathbb{R}$ is continuous, so the mapping $g$ restricted to $A_g \times A_g$ is continuous and hence bounded. So there exists $M>0$ such that $|g(x,y)|\leq M~\mbox{for all}~x,y \in A_g.$
Now, from equation \eqref{e1} we have,
\begin{eqnarray}
|g(p_1,p_2)|&=& |g(H(r,r_1,a_n),H(r,r_2,a_n))|\nonumber \\
&\leq& a_n |g(r,r)|+(1-a_n)|g(r_1,r_2)|\nonumber \\
&=& (1-a_n)|g(r_1,r_2)| \nonumber \\
&\leq& (1-a_n)|g(q_1,q_2)|+N(1-a_n)|g(q_2,r_1)| \nonumber \\
&\leq& (1-a_n)|g(q_1,q_2)|+N(1-a_n)M |g(q_2,p_1)|. \nonumber
\end{eqnarray}
Since $(1-a_n)>0$ and $N(1-a_n)M \geq 0$ so, the sequence of functions $\{f_{n}\}$ are topologically proximal weak contraction w.r.t $g$ for each $n\in \mathbb{N}$ and from Theorem \ref{t1} we can say, the mapping $f_n:A_g\rightarrow B_g$ has a best proximity point $p_n^{*} \in A_g$ such that $|g(p_n^{*},f_n(p_n^{*}))|=D_g(A_g,B_g)~\mbox{for all}~n\in \mathbb{N}.$ Since $A_g$ is $g$-sequentially compact, so the sequence $\{p_n^{*}\}\subset A_g$ has a subsequence $\{p_{n_{k}}^{*}\}$ which is $g$-convergent to $p^{*}\in A_g.$

Since $f(p_{n_{k}}^{*})\in B_g$ so there exists $y_{n_{k}}\in A_g$ such that $|g(y_{n_{k}},f(p_{n_{k}}^{*}))|=D_g(A,B).$ Also we have $|g(p_{n_{k}}^{*},f_{n_{k}}(p_{n_{k}}^{*}))|=D_g(A_g,B_g).$ Now
\begin{eqnarray}
D_g(A_g,B_g)&\leq& \Big|g(H(r,y_{n_{k}},a_{n_{k}}),f_{n_{k}}(p_{n_{k}}^{*}))\Big|
= \Big|g(f_{n_{k}}(p_{n_{k}}^{*}), H(r,y_{n_{k}},a_{n_{k}}))\Big| \nonumber\\
&\leq& a_{n_{k}}\Big|g(f_{n_{k}}(p_{n_{k}}^{*}),r)\Big|+(1-a_{n_{k}})\Big|g(f_{n_{k}}(p_{n_{k}}^{*}),y_{n_{k}})\Big|\nonumber \\
&=& a_{n_{k}}\Big|g(H(s,f(p_{n_{k}}^{*}),a_{n_{k}}),r)\Big|+(1-a_{n_{k}})\Big|g(H(s,f(p_{n_{k}}^{*}),a_{n_{k}}),y_{n_{k}})\Big| \nonumber \\
&\leq& a_{n_{k}}\Big\{a_{n_{k}} |g(r,s)|+(1-a_{n_{k}})|g(r,f(p_{n_{k}}^{*})|\Big\}\nonumber \\
&+&(1-a_{n_{k}})\Big\{a_{n_{k}} |g(y_{n_{k}},s)|+(1-a_{n_{k}})|g(y_{n_{k}},f(p_{n_{k}}^{*})|\Big\}\nonumber \\
&=& \Big\{a_{n_{k}}^{2}+(1-a_{n_{k}})^{2}\Big\}D_g(A,B)+a_{n_{k}}(1-a_{n_{k}})\Big\{|g(r,f(p_{n_{k}}^{*})|+|g(y_{n_{k}},s)|\Big\}\nonumber \\
&\leq& D_g(A_g,B_g).\nonumber
\end{eqnarray}
So we have, $\Big|g(H(r,y_{n_{k}},a_{n_{k}}),f_{n_{k}}(p_{n_{k}}^{*}))\Big|=D_g(A_g,B_g).$ So, $p_{n_{k}}^{*}=H(r,y_{n_{k}},a_{n_{k}}).$
Now,
\begin{eqnarray}
|g(p_{n_{k}}^{*},y_{n_{k}})|&=&|g(H(r,y_{n_{k}},a_{n_{k}}),y_{n_{k}})| \nonumber \\
&\leq& a_{n_{k}}|g(y_{n_{k}},r)|+(1-a_{n_{k}})|g(y_{n_{k}},y_{n_{k}})|\nonumber \\
&=& a_{n_{k}}|g(y_{n_{k}},r)|\rightarrow 0~\mbox{as}~k\rightarrow \infty. \nonumber
\end{eqnarray}
Here we use the fact that the sequence $a_n\rightarrow 0~\mbox{as}~n\rightarrow \infty.$ This shows that the sequence $\{y_{n_{k}}\}$ is $g$-convergent to $p^{*}\in A_g.$ As $f(p^{*})\in B_g$ so there exists $p^{**}\in A_g$ such that $|g(p^{**},f(p^{*}))|=D_g(A,B).$ Also we have $|g(y_{n_{k}},f(p_{n_{k}}^{*}))|=D_g(A,B).$ Since $f$ is topologically proximal Berinde non-expansive map w.r.t $g$ so we have,
$$|g(y_{n_{k}},p^{**})|\leq |g(p_{n_{k}}^{*},p^{*})|+N |g(y_{n_{k}},p^{*})|$$
$$\Longrightarrow |g(y_{n_{k}},p^{**})|\rightarrow 0~\mbox{as}~k\rightarrow \infty.$$
So the sequence $\{y_{n_{k}}\}$ is $g$-convergent to $p^{**}\in A_g.$ Since the limit is unique we have $p^{*}=p^{**}.$ So, $|g(p^{*},f(p^{*}))|=D_g(A,B).$ Hence the mapping $f$ has a best proximity point in $A_g.$
\end{proof}

Now we will provide an example to validate Theorem \ref{t3}.

\begin{example}
Consider $\mathbb{R}^{2}$ with the usual topology and $X=\{0\}\times [-1,1]$ with subspace topology. Let $g:X\times X\rightarrow \mathbb{R}$ be defined by $$g\Big((x,y),(u,v)\Big)=y-v,~(x,y),(u,v)\in X.$$ Then $g$ is a continuous function on $X\times X.$ Now we will show that $X$ is $g$-complete. Let $\{(0,x_n)\}$ be a $g$-Cauchy sequence in $X.$ So
$$\Big|g\Big((0,x_n),(0,x_m)\Big)\Big|\rightarrow 0~\mbox{as}~n,m\rightarrow \infty$$
$$\Longrightarrow |x_n-x_m|\rightarrow 0~\mbox{as}~n,m\rightarrow \infty.$$
So the sequence $\{x_n\}$ is a Cauchy sequence in $[-1,1].$ Since $[-1,1]$ is complete, so let $x_n\rightarrow x\in [-1,1]~\mbox{as}~n\rightarrow \infty.$ Now, $$\Big|g\Big((0,x_n),(0,x)\Big)\Big|=|x_n-x|\rightarrow 0~\mbox{as}~n,m\rightarrow \infty.$$ So the sequence $\{(0,x_n)\}$ is $g$-convergent to $(0,x)\in X.$ So $X$ is $g$-complete. Now let $A=\{0\}\times [-1,0]$ and $B=\{0\}\times [0,1].$ Then $D_g(A,B)=0.$ Now let $(0,x)\in A_g.$ Then there exists $(0,y)\in B$ such that $|g((0,x),(0,y))|=0.$ So $|x-y|=0.$ This is satisfied only by $x=0.$ This shows that $A_g=\{(0,0)\}.$ Similarly, $B_g=\{(0,0)\}.$ So, $A_g$ is compact, $g$-sequentially compact and $g$-closed.

Now let us define $H:X\times X\times [0,1]\rightarrow X$ by $$H\Big((0,y_1),(0,y_2),\beta\Big)=(0,\beta y_1+(1-\beta)y_2).$$ It can be seen that the mapping $H$ is topologically convex structure on $X$ w.r.t $g$ and the sets $A$ and $B$ are $(0,0)$-starshaped sets in $X.$ Also $\Big|g\Big((0,0),(0,0)\Big)\Big|=0=D_g(A,B).$ Also the pair $(A,B)$ is topologically semi-sharp proximinal in $X.$ Also here, $$|g((0,0),x)|+|g(y,(0,0))|=0=2D_g(A_g,B_g)$$ for all $x\in A_g, y\in B_g.$ So in this example the function $g$ satisfies all the conditions of Theorem \ref{t3}. Now define $f:A\rightarrow B$ by $$f(0,x)=(0,-x),~(0,x)\in A.$$ So, $f(0,0)=(0,0)\Longrightarrow f(A_g)\subseteq B_g.$ Let $(0,p_1),(0,p_2),(0,u_1),(0,u_2)\in A$ such that $$\Big|g\Big((0,p_1),f(0,u_1)\Big)\Big|=0$$ and  $$\Big|g\Big((0,p_2),f(0,u_2)\Big)\Big|=0.$$ These two equations imply that $p_1=-u_1$ and $p_2=-u_2.$ Now
$$\Big|g\Big((0,p_1),(0,p_2)\Big)\Big|=|p_1-p_2|=|u_1-u_2|=\Big|g\Big((0,u_1),(0,u_2)\Big)\Big|.$$ This shows that the mapping $f$ is topologically proximal Berinde non-expansive mapping with respect to $g.$ So all the conditions of Theorem \ref{t3} are satisfied. So by the Theorem \ref{t3}, the mapping $f$ has a best proximity point in $A_g.$ Here $p^{*}=(0,0)\in A_g$ is a best proximity point of $f.$ In this example the best proximity point is unique.
\end{example}
{\bf Open question.} In Theorem \ref{t3} we use the condition $|g(r,x)|+|g(y,s)|=2D_g(A_g,B_g)(A_g,B_g \neq \phi)~\mbox{for all}~x\in B_g,~y\in A_g,$ to prove the existence of best proximity points for topologically proximal Berinde non-expansive mapping w.r.t $g.$ Can Theorem \ref{t3} be proved without this condition?

\begin{Acknowledgement}
The first and third named authors are thankful to CSIR, Government of India, for their financial support (Ref. No. $25(0285)/18/$EMR-II).
\end{Acknowledgement}


\begin{thebibliography}{10}

\bibitem{BS}
S.S. Basha,
\newblock Best proximity points: optimal solutions
\newblock {\em J. Optim. Theory Appl.}, 151(1), 210-216 (2011)

\bibitem{BSS}
N. Bunlue and S. Suantai,
\newblock Best proximity point for proximal Berinde nonexpansive mappings on starshaped sets.
\newblock {\em Arch. Math. (Brno)}, 54,  165-176 (2018)

\bibitem{MG}
M. Gabeleh,
\newblock Proximal weakly contractive and proximal nonexpansive non-self mappings in metric and Banach spaces.
\newblock {\em J. Optim. Theory Appl.}, 158(2), 615-625 (2013)

\bibitem{MG1}
M. Gabeleh,
\newblock Best proximity point theorems via proximal non-self mappings.
\newblock {\em J. Optim. Theory Appl.}, 164,  565-576 (2015)

\bibitem{MG2}
M. Gabeleh,
\newblock Best proximity points for weak proximal contractions.
\newblock {\em Bull. Malays. Math. Sci. Soc.}, 38(1), 143-154 (2015)

\bibitem{RP}
V. Sankar Raj and T. Piramatchi,
\newblock Best proximity point theorems in topological spaces.
\newblock {\em J. Fixed Point Theory Appl.}, 22(1), Article 2 (2020)

\bibitem{TA}
W.A. Takahashi,
\newblock A convexity in metric spaces and nonexpansive mapping I.
\newblock {\em Kodai Math Semin Rep }, 22, 142-149 (1970)

\end{thebibliography}
\end{document}